\documentclass[10pt,reqno]{amsart}
\usepackage{amsmath,amscd,amsfonts,amsthm,amssymb,latexsym,mathrsfs}
\usepackage{url,color}

\usepackage{hyperref}
\hypersetup{
    bookmarks=true,         
    unicode=false,          
    pdftoolbar=true,        
    pdfmenubar=true,        
    pdffitwindow=false,     
    pdfstartview={FitH},    
    pdftitle={My title},    
    pdfauthor={Author},     
    pdfsubject={Subject},   
    pdfcreator={Creator},   
    pdfproducer={Producer}, 
    pdfkeywords={keyword1} {key2} {key3}, 
    pdfnewwindow=true,      
    colorlinks=true,       
    linkcolor=blue,          
    citecolor=blue,        
    filecolor=blue,      
    urlcolor=black    
    }

\theoremstyle{plain}
   \newtheorem{theorem}{Theorem}[section]
   
   \newtheorem{lemma}[theorem]{Lemma}
   \newtheorem{corollary}[theorem]{Corollary}

\theoremstyle{definition}
   \newtheorem{definition}{Definition}[section]

\theoremstyle{remark}
   \newtheorem{remark}[theorem]{Remark}

\numberwithin{equation}{section}

\newcommand{\lma}{\lambda_{\rm max}}
\newcommand{\lmi}{\lambda_{\rm min}}

\newcommand{\zz}{\mathbf{z}}
\newcommand{\ww}{\mathbf{w}}
\newcommand{\ee}{\mathbf{e}}

\newcommand{\xx}{\mathbf{x}}
\newcommand{\yy}{\mathbf{y}}
\newcommand{\uu}{\mathbf{u}}
\renewcommand{\ss}{\mathbf{s}}
\newcommand{\vv}{\mathbf{v}}
\newcommand{\one}{\mathbf{1}}

\newcommand{\VV}{\mathbf{V}}

\newcommand{\EE}{\mathbb{E}}
\newcommand{\PP}{\mathbb{P}}
\newcommand{\FF}{\mathcal{F}}

\newcommand{\RR}{\mathbb{R}}

\newcommand{\XR}{\mathsf{X}}

\def\newop#1{\expandafter\def\csname #1\endcsname{\mathop{\rm
#1}\nolimits}}
\newop{diag}
\newop{Spec}
\newop{Tr}
\newop{tr}
\newop{rk}
\newop{sign}

\title[Hyperbolic Kadison-Singer]{Hyperbolic polynomials and the \\ Kadison-Singer problem } 

\author{Petter Br\"and\'en}
\thanks{The author is a Wallenberg Academy Fellow
  supported by a grant from the Knut and Alice Wallenberg
  Foundation. The research is also supported by Vetenskapsr\aa det.}
\address{Department of Mathematics, KTH, Royal Institute of Technology, SE-100 44 Stockholm,
Sweden}
\email{pbranden@kth.se}


\begin{document}
\begin{abstract}
Recently Marcus,  Spielman and  Srivastava gave a spectacular proof of a theorem which implies a positive solution to the Kadison-Singer problem via Weaver's $KS_r$ conjecture. 
 We extend this theorem to the realm 
 of hyperbolic polynomials and hyperbolicity cones, as well as to arbitrary ranks.  We also sharpen the theorem by providing better bounds, which imply better bounds in Weaver's $KS_r$ conjecture for each $r>2$. For $r=2$ our bound agrees with Bownik \emph{et al.} \cite{R2}. 
 
 
\end{abstract}
\maketitle

\tableofcontents

\thispagestyle{empty}

\bigskip \bigskip \bigskip \bigskip \bigskip \bigskip \bigskip \bigskip \bigskip 

\small{ \noindent 
Parts of this work are based on unpublished notes \cite{BLec0,BLec} from a graduate course focused on hyperbolic polynomials and the papers \cite{MSS1,MSS2} of Marcus,  Spielman and  Srivastava, given by the author at the Royal Institute of Technology (Stockholm) in the fall of 2013.}

\newpage
\section{Introduction and main result}
The following theorem of Marcus, Spielman and Srivastava is a stronger version of Weaver's $KS_r$ conjecture \cite{We}, which implies a positive solution to the Kadison-Singer problem \cite{KS}. See \cite{Cas} for a review of the many consequences of Theorem~\ref{MSSmain}.

\begin{theorem}[Marcus,  Spielman and  Srivastava \cite{MSS2}]\label{MSSmain}
Let $r \geq 2$ be an integer and $\epsilon$ a positive real number. Suppose $A_1, \ldots, A_m$ are positive semidefinite hermitian rank at most one  matrices of size $d \times d$ satisfying 
$$\sum_{i=1}^m A_i = I, $$ where $I$ is the identity matrix. If  $\tr(A_i) \leq \epsilon$ for all $1\leq i \leq m$, then  there is a partition $S_1\cup S_2 \cup \cdots \cup S_r=[m]:=\{1,2,\ldots,m\}$ such that 
\begin{equation}\label{sqbound1}
\left\| \sum_{i \in S_j} A_i  \right\| \leq \frac {(1+\sqrt{r\epsilon})^2} r,    
\end{equation}
for each $j \in [r]$, where $\|\cdot \|$ denotes the operator matrix norm, and $\tr(A)$ denotes the trace of $A$. 
\end{theorem}

One purpose of this work is to extend Theorem~\ref{MSSmain}  to hyperbolic polynomials and hyperbolicity cones. A benefit of the extension (Theorem~\ref{t1})  is that the proof becomes coherent in its general form, and fits naturally in the theory of hyperbolic polynomials.  In particular we don't need to use the Helton-Vinnikov theorem to translate between matrices and hyperbolic polynomials. In our more general setting Theorem~\ref{t1} applies to e.g. hermitian matrices over quaternions, Euclidean Jordan Algebras and Symmetric Domains, see \cite{AB}. We also get rid of the rank constraints in Theorem~\ref{MSSmain}, which was independently achieved by Michael Cohen \cite{Cohen} for complex hermitian matrices. The other main purpose of this work is to sharpen the inequalities in \eqref{sqbound1}. For $r=2$ our upper bound coincides with that of \cite{R2}, while for each $r>2$ we provide better bounds than previously known.  

The plan of the paper is as follows. In this section we provide relevant background on hyperbolic polynomials and state our main results. In Section \ref{scompatible} we consider compatible families of polynomials. This is a subclass of the class of interlacing families considered in \cite{MSS1,MSS2}. In Section \ref{smixed} we define mixed hyperbolic polynomials and construct a large class of compatible families of polynomials arising from mixed hyperbolic polynomials. In Section \ref{scorr} we derive inequalities for hyperbolic polynomials needed to prove Theorem~\ref{t1}. In Section \ref{sboundmix} we define the mixed characteristic polynomial of a tuple of vectors in the hyperbolicity cone. We use the inequalities derived in Section \ref{scorr} to find upper bounds for the largest zero of a mixed characteristic polynomial. In Section \ref{sproof} we use the results in the previous sections to prove our main theorem, Theorem~\ref{t1}. Finally, in Section \ref{sspec} we prove Theorem~\ref{specifics} which provides explicit bounds in  Theorem~\ref{t1}. \\[1ex]

Hyperbolic polynomials are multivariate generalizations of real-rooted polynomials and determinants. They have their  origin  in PDE theory where they were studied by e.g. Petrovsky, G\aa rding, Bott, Atiyah and H\"ormander, see \cite{ABG,Ga,Horm}.  During recent years  hyperbolic polynomials have been studied in diverse areas such as 
control theory, optimization, real algebraic geometry, probability theory, computer science and combinatorics, see  \cite{BBL,Pem,Ren,Vin,Wag} and the references therein. 

A homogeneous polynomial $h(\xx) \in \RR[x_1, \ldots, x_n]$ is \emph{hyperbolic} with respect to a vector 
$\ee \in \RR^n$ if $h(\ee) \neq 0$, and if for all $\xx \in \RR^n$ the univariate polynomial $t \mapsto h(t\ee -\xx)$ has only real 
zeros. Here are some examples of hyperbolic polynomials:
\begin{enumerate}
\item Let $h(\xx)= x_1x_2\cdots x_n$. Then $h(\xx)$ is hyperbolic with respect to any vector $\ee \in \RR_{++}^n=(0,\infty)^n$: 
$$
h(t\ee-\xx) = \prod_{j=1}^n (te_j-x_j).
$$
\item Let $X=(x_{ij})_{i,j=1}^n$ be a matrix of $n(n+1)/2$ variables where we impose $x_{ij}=x_{ji}$ for all $1\leq i<j\leq n$. Then $\det(X)$ is hyperbolic with respect to the identity matrix 
$I=\diag(1, \ldots, 1)$. Indeed $t \mapsto \det(tI-X)$ is the characteristic polynomial of the symmetric matrix $X$, so it has only real zeros.

More generally we may consider complex hermitian matrices $Z=(x_{jk}+iy_{jk})_{j,k=1}^n$ (where $i = \sqrt{-1}$) of $n^2$ real variables where we impose $x_{jk}=x_{kj}$ and $y_{jk}=-y_{kj}$, for all $1\leq j\leq k \leq n$.  Then $\det(Z)$ is a real polynomial which is hyperbolic with respect to $I$.
\item Let $h(\xx)=x_1^2-x_2^2-\cdots-x_n^2$. Then $h$ is hyperbolic with respect to $(1,0,\ldots,0)^T$. 
\end{enumerate}

Suppose $h$ is hyperbolic with respect to $\ee \in \RR^n$. We may write 
\begin{equation}\label{dalambdas}
h(t\ee-\xx) = h(\ee)\prod_{j=1}^d (t - \lambda_j(\xx)),
\end{equation}
where $\lma(\xx)=\lambda_1(\xx) \geq \cdots \geq \lambda_d(\xx)=\lmi(\xx)$ are called the \emph{eigenvalues} of $\xx$ (with respect to $\ee$), and $d$ is the degree of $h$. In particular 
\begin{equation}\label{prolambda}
h(\xx) = h(\ee)\lambda_1(\xx) \cdots \lambda_d(\xx). 
\end{equation}

By homogeneity 
\begin{equation}\label{dilambdas}
\lambda_j(s\xx+t\ee)= 
\begin{cases}
s\lambda_j(\xx)+t, &\mbox{ if } s\geq 0 \mbox{ and } \\
s\lambda_{d-j}(\xx)+t, &\mbox{ if } s \leq 0
\end{cases},
\end{equation}
for all $s,t \in \RR$ and $\xx \in \RR^n$.

The (open) \emph{hyperbolicity cone} of $h$ with respect to $\ee$ is the set 
$$
\Lambda_{\tiny{++}}(h,\ee)= \{ \xx \in \RR^n : \lmi(\xx) >0\}.
$$
We sometimes abbreviate and write $\Lambda_{\tiny{++}}(\ee)$ or $\Lambda_{\tiny{++}}$. 
We denote its closure by $ \Lambda_{\tiny{+}}= \Lambda_{\tiny{+}}(h,\ee)=\{ \xx \in \RR^n : \lmi(\xx) \geq 0\}$. 
Since $h(t\ee-\ee)=h(\ee)(t-1)^d$, we see that $\ee \in \Lambda_{\tiny{++}}$. The hyperbolicity cones for the examples  above are:
\begin{enumerate}
\item $\Lambda_{\tiny{++}}(\ee)= \RR_{++}^n$. 
\item $\Lambda_{\tiny{++}}(I)$ is the cone of positive definite matrices.
\item $\Lambda_{\tiny{++}}(1,0,\ldots,0)$  is the \emph{Lorentz cone} 
$$
\left\{\xx \in \RR^n : x_1 > \sqrt{x_2^2+\cdots+x_n^2}\right\}.
$$
\end{enumerate}

The following theorem collects a few fundamental facts about hyperbolic polynomials and their hyperbolicity cones. For proofs see \cite{Ga,Ren}. 

\begin{theorem}[G\aa rding, \cite{Ga}]\label{hypfund}
Suppose $h$ is hyperbolic with respect to $\ee \in \RR^n$. 
\begin{enumerate}
\item $\Lambda_+(\ee)$ and $\Lambda_{++}(\ee)$ are convex cones.
\item $\Lambda_{++}(\ee)$ is the connected component of 
$$
\{ \xx \in \RR^n : h(\xx) \neq 0\}
$$
which contains $\ee$. 
\item $\lmi : \RR^n \rightarrow \RR$ is a concave function, and  $\lma : \RR^n \rightarrow \RR$ is a convex function. 
\item If $\ee' \in \Lambda_{++}(\ee)$, then $h$ is hyperbolic with respect to $\ee'$ and $\Lambda_{++}(\ee')=\Lambda_{++}(\ee)$.
\end{enumerate}
\end{theorem} 

Recall that the \emph{lineality space}, $L(C)$, of a convex cone $C\subseteq \RR^n$ is $C \cap (-C)$, i.e., the largest linear space contained in $C$. It follows that $L(\Lambda_+)= \{\xx \in \RR^n : \lambda_i(\xx)=0 \mbox{ for all } i\}$, see e.g. \cite{Ren}.

The \emph{trace}, \emph{rank} and \emph{spectral radius} (with respect to $\ee$) of  $\xx \in \RR^n$ are defined as for matrices: 
$$
\tr(\xx) = \sum_{i=1}^d\lambda_i(\xx), \ \ \rk(\xx)= \#\{ i : \lambda_i(\xx)\neq 0\} \ \  \mbox{ and } \ \  \|\xx\| = \max_{1\leq i\leq d} |\lambda_i(\xx)|.
$$
Note that $\| \xx \| = \max\{ \lma(\xx), -\lmi(\xx)\}$ and hence $ \| \cdot \|$ is convex by Theorem~\ref{hypfund}~(3). It follows that $\| \cdot \|$ is a seminorm, and that $\| \xx \|=0$ if and only if $\xx \in L(\Lambda_+)$. Hence $\| \cdot \|$ is a norm if and only if  $L(\Lambda_+)=\{0\}$. 

For a positive integer $r$, let $U_r$ be the set of all pairs $(\delta,\mu)$ of positive real numbers such that 
\begin{equation}\label{dura}
\delta-1 \geq \frac \delta \mu \cdot \frac { \left(1+ \frac \delta {r \mu} \right)^{r-1}- \left(\frac \delta {r \mu} \right)^{r-1}} { \left(1+ \frac \delta {r \mu} \right)^{r}- \left(\frac \delta {r \mu} \right)^{r}},
\end{equation}
and either 
\begin{itemize}
\item $\mu >1$, or 
\item $1\leq \delta \leq 2$ and $\mu >1-\delta/r$. 
\end{itemize}
It is not hard to see that $U_1 \supseteq U_2 \supseteq U_3 \supseteq \cdots \supseteq U_\infty$,
where $U_\infty = \{(\delta,\mu) \in \RR_{++}^2 : \mu \geq \delta /(\delta-1) \}$. 
Let further
$$
\delta(\epsilon, m,r) = \inf \left\{ \frac {\epsilon \mu+ \left(1-\frac 1 m\right)\delta} {1-\frac 1 m + \frac \mu m }  : (\delta, \mu) \in U_r \right\}.
$$
For applications we often want a bound which is independent of $m$. We define 
$$
 \delta(\epsilon, \infty,r) = \inf \left\{ \epsilon \mu+ \delta  : (\delta, \mu) \in U_r \right\}, \ \ \ \ r=1,2,3,\ldots, \infty.
$$
The definition of $\delta(\epsilon, m,r)$ may seem obscure. To make sense of it we will compute it explicitly for some cases and give upper bounds (Theorem~\ref{specifics}). 
The following theorem generalizes Theorem~\ref{MSSmain} to hyperbolic polynomials. It applies to arbitrary ranks as well as improves the bound in \eqref{sqbound1}. 
\begin{theorem}\label{t1}
Let $k\geq 2$ be an integer and $\epsilon$ a positive real number. Suppose $h$ is hyperbolic  with respect to $\ee \in \RR^n$, and let $\uu_1, \ldots, \uu_m \in \Lambda_{+}(h,\ee)$ be such that 
\begin{itemize}
\item[] $\tr(\uu_i) \leq \epsilon$ for all $1\leq i \leq m$, 
\item[] $\rk(\uu_i) \leq r$ for all $1\leq i \leq m$, and
\item[] $\uu_1+ \uu_2+\cdots+ \uu_m=\ee$. 
\end{itemize}
Then there is a partition $S_1\cup S_2 \cup \cdots \cup S_k=[m]$ such that 
\begin{equation}\label{sqbound}
\left\| \sum_{i \in S_j} \uu_i  \right\| \leq \frac 1 k \delta(k\epsilon, m,rk),    
\end{equation}
for each $j \in [k]$. 
\end{theorem}
Note that we recover Theorem~\ref{MSSmain} if we let $h= \det$ in Theorem~\ref{t1}, combined with \eqref{a1} below.  Theorem~\ref{t1} combined with \eqref{a2} extends the main result of \cite{R2} to hyperbolic polynomials. Theorem~\ref{t1} combined with \eqref{a3} produces better bounds in Theorem~\ref{MSSmain} for all $r >2$.   
\begin{theorem}\label{specifics}
\begin{align}
\delta(\epsilon,m, \infty) &= \left( 1-\frac 1 m +\sqrt{\epsilon - \frac 1 m \left(1-\frac 1 m\right)}\right)^2,  \label{a0} \\ 
\delta(\epsilon,\infty, \infty) &= (1+\sqrt{\epsilon})^2 =1+2\sqrt{\epsilon}+\epsilon, \mbox{ for all } \epsilon>0, \label{a1}\\
\delta(\epsilon, \infty, 2) &= 
\begin{cases}
1 + 2\sqrt{\epsilon}\sqrt{1-\epsilon}, &\mbox{ if } 0<\epsilon \leq 1/2 \\
2, &\mbox{ if } \epsilon >1/2. 
\end{cases}\label{a2}\\
\delta(\epsilon, \infty, r) &\leq \begin{cases}
1+2\sqrt{\epsilon}\sqrt{1-\epsilon/r}+ \frac {r-1} r \epsilon, &\mbox{ if } \epsilon \leq r/(r+1), \\ 
2+\epsilon(1-2/r), &\mbox{ if } \epsilon > r/(r+1).
\end{cases}\label{a3}
 \end{align}
\end{theorem}

\section{Compatible families of polynomials}\label{scompatible}
We say that a univariate real polynomial $f$ is \emph{real-rooted} if either $f \equiv 0$, or $f$ has only real zeros. 
Let $f$ and $g$ be two real-rooted polynomials of degree $d-1$ and $d$, respectively, where $d\geq 1$. We say that $f$ is an \emph{interleaver} of $g$ if 
$$
\beta_1 \leq \alpha_1\leq \beta_2 \leq \alpha_2 \leq \cdots \leq \alpha_{d-1} \leq \beta_d, $$ 
where $\alpha_1 \leq \cdots \leq \alpha_{d-1}$ and $\beta_1 \leq \cdots \leq \beta_{d}$ are the zeros of $f$ and $g$, respectively. 

A family of polynomials $\{f_1(x), \ldots, f_m(x)\}$ of real-rooted polynomials of the same degree and the same sign of leading coefficients  is called \emph{compatible} if it satisfies any of the equivalent conditions in the next theorem. Theorem~\ref{CS} has been discovered several times. We refer to \cite[Theorem 3.6]{CS} for a proof. 
\begin{theorem}\label{CS}
Let $f_1, \ldots, f_m$ be real-rooted polynomials of the same degree and with positive leading coefficients. The following are equivalent. 
\begin{enumerate}
\item $f_1, \ldots, f_m$ have a common interleaver, i.e., there is a polynomial $g$ which is an interleaver of each $f_i$, $1\leq i \leq m$. 
\item for all $p_1, \ldots, p_m \geq 0$, $\sum_{i}p_i=1$, the polynomial
$$
p_1f_1+ \cdots+ p_mf_m
$$
is real-rooted. 
\end{enumerate}
\end{theorem}

\begin{lemma}[\cite{MSS1}]\label{largestz}
Let $f_1,\ldots, f_m$ be real-rooted polynomials 
that have the same degree and positive leading coefficients, and suppose $p_1, \ldots, p_m \geq 0$ sum to one.  If $\{f_1,\ldots, f_m\}$ is compatible, then for some $1 \leq i \leq m$ with $p_i >0$ the largest zero of $f_i$ is smaller or equal to the largest zero of the polynomial 
$$
 f=p_1f_1 + p_2f_2 + \cdots + p_mf_m.
$$\end{lemma}
\begin{proof}
If $\alpha$ is the largest zero of the common interleaver $g$, then $f_i(\alpha) \leq 0$ for all $i$.  Hence the largest zero, $\beta$,  of $f$ is located in the interval $[\alpha, \infty)$, as are the largest zeros of $f_i$ for each $1\leq i \leq m$. Since $f(\beta)=0$, there is an index $i$ with $p_i >0$ such that $f_i(\beta) \geq 0$. Hence the largest zero of $f_i$ is at most $\beta$. 
\end{proof}

The next definition may be seen as a generalization of compatible polynomials to arrays. 

\begin{definition}
Let $S_1, \ldots, S_m$ be finite sets. A family, $$\FF=\{f(\ss;t)\}_{\ss \in S_1 \times \cdots \times S_m} \subset \RR[t],$$ of polynomials  is called \emph{compatible}  if 
\begin{itemize}
\item all non-zero members of $\FF$ have the same degree and the same signs of their leading coefficients, and 
\item for all choices of independent random variables 
$\XR_1 \in S_1, \ldots, \XR_m \in S_m$, the polynomial 
$
\EE f(\XR_1,\ldots, \XR_n;t) 
$
is real-rooted. 
\end{itemize}
\end{definition}

The notion of a compatible family of polynomials is less general than that of an \emph{interlacing family of polynomials} in \cite{MSS1,MSS2}. However since all families appearing here (and in \cite{MSS1,MSS2}) are compatible, we find it more convenient to work with these. The following theorem is in essence from \cite{MSS1}. 

\begin{theorem}\label{expfam}
Let $\{f(\ss;t)\}_{\ss \in S_1 \times \cdots \times S_m}$ be a compatible family, and let \ $\XR_1 \in S_1, \ldots, \XR_m \in S_m$ be independent random variables such that $\EE f(\XR_1,\ldots, \XR_m;t) \not \equiv 0$. Then there is a tuple $\ss=(s_1, \ldots, s_n) \in S_1 \times \cdots \times S_m$, with $\PP[\XR_i=s_i]>0$ for each $1\leq i \leq m$, such that the largest zero of $f(s_1,\ldots, s_m;t)$ is smaller or equal to the largest zero of 
$\EE f(\XR_1,\ldots, \XR_m;t)$. 
\end{theorem}

\begin{proof}
The proof is by induction over $m$. The case when $m=1$ is Lemma~\ref{largestz}, so suppose $m>1$.  If 
$S_m=\{c_1,\ldots, c_k\}$, then  
$$
\EE f(\XR_1,\ldots, \XR_m;t)= \sum_{i=1}^k q_i \EE f(\XR_1,\ldots, \XR_{m-1}, c_i;t), 
$$
for some $q_i \geq 0$. However 
$$
\sum_{i=1}^k p_i \EE f(\XR_1,\ldots, \XR_{m-1}, c_i;t)
$$
is real-rooted for all choices of $p_i \geq 0$ such that $\sum_ip_i=1$. By Lemma~\ref{largestz} and Theorem~\ref{CS} there is an index $j$ with $q_j>0$ such that $\EE f(\XR_1,\ldots, \XR_{m-1}, c_j;t) \not \equiv 0$ and such that the largest zero of $\EE f(\XR_1,\ldots, \XR_{m-1}, c_j;t)$ is no larger than the largest zero of $\EE f(\XR_1,\ldots, \XR_m;t)$. The theorem now follows by induction. 
 \end{proof}

 \section{Mixed hyperbolic polynomials}\label{smixed}
 In this section we will produce a large class of compatible families of polynomials arising from (mixed) hyperbolic polynomials. 
  
 Recall that the \emph{directional derivative} of $h(\xx) \in  \RR[x_1,\ldots, x_n]$ with respect to $\vv =(v_1,\ldots, v_n)^T \in \RR^n$ is defined as 
 $$
 D_\vv h(\xx) := \sum_{k=0}^n v_k \frac{ \partial h }{\partial x_k}(\xx),
 $$
 and note that 
 \begin{equation}\label{dvalt}
  (D_\vv h)(\xx+t\vv) = \frac d {dt} h(\xx+ t \vv) . 
\end{equation}
If $h$ is hyperbolic with respect to $\ee$, then 
$$
\tr(\vv)= \frac {D_\vv h(\ee)}{h(\ee)},
$$
by \eqref{dalambdas}. Hence $\vv \rightarrow \tr(\vv)$ is linear.

The following theorem is essentially known, see e.g. \cite{BGLS,Ga,Ren}. However we need slightly more general results, so we provide proofs below, when necessary. 
\begin{theorem}\label{direct}
Let $h$ be a hyperbolic polynomial and let $\vv \in \Lambda_+$ be such that   $D_\vv h \not \equiv 0$. Then 
\begin{enumerate}
\item $D_\vv h$ is hyperbolic with hyperbolicity cone containing $\Lambda_{++}$.  
\item The polynomial $h(\xx)-yD_\vv h(\xx) \in \RR[x_1,\ldots, x_n,y]$ is hyperbolic with hyperbolicity cone containing $\Lambda_{++} \times \{y: y \leq 0\}$. 
\item The rational function 
$$
\xx \mapsto \frac {h(\xx)}{D_\vv h(\xx)}
$$
is concave on $\Lambda_{++}$. 
\end{enumerate}
\end{theorem}
\begin{proof}
(1). See \cite[Lemma 4]{BrOp}.

(2). The polynomial $h(\xx)y$ is hyperbolic with hyperbolicity cone containing $\Lambda_{++} \times \{y : y<0\}$. Hence so is $H(\xx,y):= -D_{(\vv,-1)} h(\xx)y= h(\xx)- y D_\vv h(\xx)$ by (1). Since $H(\ee',0) = h(\ee') \neq 0$ for each $\ee' \in \Lambda_{++}$, we see that also $\Lambda_{++}\times \{0\}$ is a subset of the hyperbolicity cone (by Theorem~\ref{hypfund} (2)) of $H$.

(3). If $\xx \in \Lambda_{++}$, then (by Theorem~\ref{hypfund} (2)) $(\xx,y)$ is in the closure of the hyperbolicity cone of $H(\xx,y)$ if and only if 
$$
y \leq \frac {h(\xx)}{D_\vv h(\xx)}.
$$
Since hyperbolicity cones are convex, 
$$
y_1 \leq \frac {h(\xx_1)}{D_\vv h(\xx_1)} \mbox{ and } y_2 \leq \frac {h(\xx_2)}{D_\vv h(\xx_2)} \mbox{ imply } y_1+y_2 \leq \frac {h(\xx_1+\xx_2)}{D_\vv h(\xx_1+\xx_2)}, 
$$
for all $\xx_1,\xx_2 \in \Lambda_{++}$, 
from which (3) follows. 
\end{proof}

\begin{lemma}\label{rankalt}
Let $h$ be hyperbolic with hyperbolicity cone $\Lambda_{++}\subseteq \RR^n$. The rank function does not depend on the choice of $\ee \in \Lambda_{++}$, and 
$$
\rk(\vv)= \max\{ k : D_\vv^kh \not \equiv 0\}, \quad \mbox{ for all } \vv \in \RR^n.
$$
\end{lemma}
\begin{proof}
That the rank does not depend on the choice of $\ee \in \Lambda_{++}$ is known, see \cite[Prop. 22]{Ren} or \cite[Lemma 4.4]{BrObs}.

By \eqref{dvalt}
\begin{equation}\label{mag}
h(\xx-y\vv) = \left( \sum_{k=0}^{\infty} \frac {(-y)^k D_\vv^k}{k!} \right) h(\xx). 
\end{equation}
 Thus  
$$
h(\ee-t\vv) = h(\ee)\prod_{j=1}^d(1-t\lambda_j(\vv))= \sum_{k=0}^d (-1)^k\frac {D^k_\vv h(\ee)} {k!} t^k, 
$$
 and hence $\rk(\vv)= \deg h(\ee-t\vv)=  \max\{k : D^k_\vv h(\ee)\neq 0\}$. Since the rank does not depend on the choice of $\ee \in \Lambda_{++}$, if $D^{k+1}_\vv h(\ee)=D^{k+2}_\vv h(\ee)=\cdots =0$ for some 
$\ee \in \Lambda_{++}$, then $D^{k+1}_\vv h(\ee')=D^{k+2}_\vv h(\ee')=\cdots =0$  for all $\ee' \in \Lambda_{++}$. Since $ \Lambda_{++}$ has non-empty interior this means $D^{k+1}_\vv h \equiv 0$.
\end{proof}

If $h(\xx) \in \RR[x_1,\ldots, x_n]$ and $\vv_1, \ldots, \vv_m \in \RR^n$, let $h[\vv_1, \ldots, \vv_m]$ be the polynomial in $\RR[x_1,\ldots, x_n,y_1,\ldots, y_m]$ defined by 
$$
h[\vv_1, \ldots, \vv_m] = \prod_{j=1}^m \left(1-y_jD_{\vv_j}\right) h(\xx).
$$
We call $h[\vv_1, \ldots, \vv_m]$ a \emph{mixed hyperbolic polynomial}. 
By iterating Theorem~\ref{direct} (2) we get: 
\begin{theorem}\label{mixhyp}
If $h(\xx)$ is hyperbolic with hyperbolicity cone $\Lambda_{++}$ and $\vv_1, \ldots, \vv_m \in \Lambda_+$, then $h[\vv_1, \ldots, \vv_m]$ is hyperbolic with hyperbolicity cone containing $\Lambda_{++} \times (-\RR_+^m)$, where $\RR_+ := [0,\infty)$. 
\end{theorem}

\begin{lemma}\label{rk1le}
Suppose $h$ is hyperbolic. If $\vv_1, \ldots, \vv_m \in \Lambda_+$ have rank at most one, then 
$$
h[\vv_1, \ldots, \vv_m] = h(\xx-y_1\vv_1 - \cdots - y_m \vv_m).
$$
\end{lemma}

\begin{proof}
If $\vv$ has rank at most one, then $D_\vv^k h \equiv 0$ for all $k \geq 2$ by Lemma~\ref{rankalt}. Hence, by \eqref{mag}, 
$$
h(\xx-y\vv) = \left( \sum_{k=0}^{\infty} \frac {(-y)^k D_\vv^k}{k!} \right) h(\xx)= (1-yD_{\vv})h(\xx),
$$
from which the lemma follows.
\end{proof}

Note that $(\vv_1,\ldots,\vv_m) \mapsto h[\vv_1,\ldots,\vv_m]$ is affine linear in each coordinate, i.e., for all $p \in \RR$ and $1\leq i \leq m$:
\begin{align*}
& h[\vv_1,\ldots,(1-p)\vv_i+p\vv_i',\ldots, \vv_m] \\
= &(1-p)h[\vv_1,\ldots,\vv_i,\ldots, \vv_m] +ph[\vv_1,\ldots,\vv_i',\ldots, \vv_m].
\end{align*}
Hence if $\XR_1, \ldots, \XR_m$ are  independent random variables in $\RR^n$, then 
\begin{equation}\label{mixedexp}
\EE h[\XR_1,\ldots,\XR_m] = h[\EE \XR_1,\ldots,\EE \XR_m]. 
\end{equation}

The next theorem provides examples of compatible families of polynomials. 

\begin{theorem}\label{mixedchar}
Let $h(\xx)$ be hyperbolic with respect to $\ee\in \RR^n$, let $V_1, \ldots, V_m$ be finite sets of vectors  in $\Lambda_+$, and let $\ww \in \RR^{n+m}$.  For $\VV =(\vv_1,\ldots, \vv_m) \in  V_1\times \cdots \times V_m$, let 
$$f(\VV;t) := h[\vv_1,\ldots, \vv_m](t\ee +\ww).
$$ 
Then $\{f(\VV;t)\}_{\VV \in V_1\times \cdots \times V_m}$ is a compatible family.

In particular if in addition all vectors in $V_1 \cup \cdots \cup V_m$ have rank at most one, and 
$$
g(\VV;t) := h(t\ee +\ww- \alpha_1\vv_1-\cdots-\alpha_m\vv_m),
$$
where $\ww \in \RR^n$  and $(\alpha_1,\ldots, \alpha_m)\in \RR^m$, then $\{g(\VV;t)\}_{\VV \in V_1\times \cdots \times V_m}$ is a compatible family.
\end{theorem}

\begin{proof}
Let $\XR_1 \in V_1, \ldots, \XR_m \in V_m$ be independent random variables. Then the polynomial 
$\EE h[\XR_1, \ldots, \XR_m]= h[\EE \XR_1,\ldots,\EE \XR_m]$ is hyperbolic with respect to $(\ee, 0,\ldots,0)$ by Theorem~\ref{mixhyp} (since $\EE \vv_i \in \Lambda_+$ for all $i$ by convexity). In particular the polynomial $\EE f(\XR_1,\ldots, \XR_m;t)$ is real-rooted.

 The second assertion is an immediate consequence of the first combined with Lemma~\ref{rk1le}.

\end{proof}

\section{Correlation inequalities for hyperbolic polynomials}\label{scorr}
In this (technical) section we will derive inequalities for hyperbolic polynomials needed to prove the bound in Theorem~\ref{t1}. 

A consequence of Theorem~\ref{direct} (3) is the correlation inequality 
\begin{equation}\label{corn}
D_\uu h(\xx) \cdot D_\vv h(\xx) - D_\uu D_\vv h(\xx) \cdot h(\xx) \geq 0,
\end{equation}
for all hyperbolic polynomials $h$ and $\uu,\vv, \xx \in \Lambda_+$, see e.g. \cite[Section 3]{BBL}. Indeed 
$$
D_\uu h(\xx) \cdot D_\vv h(\xx) - D_\uu D_\vv h(\xx) \cdot h(\xx)=
(D_\uu h)(\xx)^2\cdot D_\vv \left(\frac h {D_\uu h}\right)(\xx) \geq 0, 
$$
by concavity (Theorem~\ref{direct}).

We also have the higher correlation inequalities   
\begin{equation}\label{highcorn}
D^k_\uu h(\xx) \cdot D_\vv h(\xx) - D^k_\uu D_\vv h(\xx) \cdot h(\xx)\geq 0
\end{equation}
for all $\uu,\vv, \xx \in \Lambda_+$ and $k \geq 0$. Indeed 
\begin{align*}
D^k_\uu h(\xx) \cdot D_\vv h(\xx) &- D^k_\uu D_\vv h(\xx) \cdot h(\xx) = -h(\xx)^2\cdot D_\vv \left(\frac {D^k_\uu h} h\right)(\xx) \\
&= 
-h(\xx)^2\cdot D_\vv \left(\frac {D^{k}_\uu h} {D^{k-1}_\uu h} \frac {D^{k-1}_\uu h} {D^{k-2}_\uu h}\cdots 
\frac {D_\uu h} h\right)(\xx) \geq 0, 
\end{align*}
by Leibniz' rule and \eqref{corn} for the hyperbolic polynomials $g=D_\uu^jh$.  We want to relate the quantities in the left hand side of \eqref{highcorn} for different $k$. For the rest of this section we fix a hyperbolic  polynomial $h$, and vectors $\uu,\vv \in \Lambda_+$ and $\xx \in \Lambda_{++}$. 
To enhance readability in the computations to come, let 
\begin{align*}
\Phi_k & = -D_\vv \left(\frac {D_\uu^k h} h\right)(\xx), \ \ \mbox{ and }\\
\eta_k &= \frac {D_\uu^k h} h(\xx),
\end{align*}
for all $k \geq 0$. Note that $\eta_k >0$ for $0\leq k \leq \rk(\uu)$, and $\eta_k = 0$ for $k>\rk(\uu)$. 
\begin{lemma}\label{t-rec}
For $1\leq k \leq \rk(\uu)+1$, 
$$
\Phi_{k+1} \leq 2 \eta_k\cdot \eta_{k-1}^{-1}\cdot\Phi_k+ 
(-2\eta_k^2\cdot\eta_{k-1}^{-2}+\eta_{k+1}\cdot\eta_{k-1}^{-1})\cdot\Phi_{k-1}.
$$
\end{lemma}
\begin{proof}
First note that 
\begin{align}\label{forsta}
D_\vv D_\uu \left(\frac {D_\uu^k h}{h} \right) &= D_\vv\left(\frac {D_\uu^{k+1} h}{h} - 
\frac {D_\uu^{k} h}{h}\cdot \frac {D_\uu h}{h}\right) \nonumber \\
&= -\Phi_{k+1}+ \eta_{1} \cdot \Phi_{k}+ \eta_{k}\cdot \Phi_{1}.
\end{align}
Also, by Leibniz' rule
\begin{align}\label{andra}
\Phi_k &= -D_\vv \left(\frac {D_\uu^k h} {D_\uu^{k-1}h} \cdot \frac {D_\uu^{k-1} h} {h}\right) \nonumber \\
&= -D_\vv \left(\frac {D_\uu^k h} {D_\uu^{k-1}h}\right) \cdot  \frac {D_\uu^{k-1} h} {h}- \frac {D_\uu^k h} {D_\uu^{k-1}h} \cdot D_\vv \left( \frac {D_\uu^{k-1} h} {h}\right) \nonumber \\
&= -D_\vv \left(\frac {D_\uu^k h} {D_\uu^{k-1}h}\right) \cdot  \eta_{k-1}+ \eta_k\cdot \eta_{k-1}^{-1} \cdot \Phi_{k-1}.
\end{align}
By Leibniz' rule again

\begin{align*}
D_\vv D_\uu \left(\frac {D_\uu^k h}{h} \right) &=  D_\vv D_\uu \left(\frac {D_\uu^k h}{D_\uu^{k-1}h} \cdot \frac {D_\uu^{k-1} h}{h} \right)\\
&= D_\vv D_\uu\left(\frac {D_\uu^k h}{D_\uu^{k-1}h}\right)\cdot \frac {D_\uu^{k-1} h}{h} + D_\vv\left(\frac {D_\uu^k h}{D_\uu^{k-1}h}\right)\cdot D_\uu\left(\frac {D_\uu^{k-1} h}{h}\right)+\\
&+D_\uu\left(\frac {D_\uu^k h}{D_\uu^{k-1}h}\right)\cdot D_\vv\left(\frac {D_\uu^{k-1} h}{h}\right)+\frac {D_\uu^k h}{D_\uu^{k-1}h}\cdot D_\vv  D_\uu\left(\frac {D_\uu^{k-1} h}{h}\right)
\end{align*}
We claim that the term 
\begin{equation}\label{daterm}
D_\vv D_\uu\left(\frac {D_\uu^k h}{D_\uu^{k-1}h}\right)\cdot \frac {D_\uu^{k-1} h}{h} 
\end{equation}
is nonnegative. The second factor is nonnegative. If $g=D_\uu^{k-1}h$ we see that the first factor equals 
$
D_\uu^2 \left({D_\vv g}/{g}\right)
$, 
which is nonnegative since ${D_\vv g}/{g}$ is convex on $\Lambda_{++}$ (Theorem~\ref{direct}).

Using \eqref{forsta}, \eqref{andra}, and the nonnegativity of \eqref{daterm} we get 
\begin{align*}
&-\Phi_{k+1}+ \eta_{1} \cdot \Phi_{k}+ \eta_{k}\cdot \Phi_{1} \\
\geq &\left(-\eta_{k-1}^{-1}\cdot\Phi_k+\eta_k\cdot \eta_{k-1}^{-2}\cdot\Phi_{k-1}\right)\cdot (\eta_k-\eta_{k-1}\cdot \eta_1)-(\eta_{k+1}\cdot \eta_{k-1} -\eta_k\cdot \eta_k)\cdot \eta_{k-1}^{-2}\cdot\Phi_{k-1} \\
+ &\eta_k\cdot \eta_{k-1}^{-1}\cdot \left(-\Phi_{k}+ \eta_{1} \cdot \Phi_{k-1}+ \eta_{k-1}\cdot \Phi_{1} \right),
\end{align*}
which simplifies to the desired inequality. 
\end{proof}
If $p(t) = \sum_{k=0}^r\binom r k a_k t^k$ is a real-rooted polynomial, then 
\emph{Newton's inequalities} \cite{Newton} say that 
$$
a_k^2 \geq a_{k-1}\cdot a_{k+1},
$$
for all $1\leq k \leq r-1$. If additionally $a_0,a_1,\ldots, a_k >  0$, then 
\begin{equation}\label{ncons}
\frac {a_k} {a_{k-1}} \leq \frac {a_{k-1}} {a_{k-2}} \leq \cdots \leq \frac {a_{1}} {a_{0}}.
\end{equation}
and thus 
\begin{equation}\label{ajbound}
a_j \leq a_0 \left( \frac {a_1} {a_0}\right)^j, \ \ \ \mbox{ for all } 0\leq j \leq k.
\end{equation}

\begin{lemma}\label{post}
Suppose $\rk(\uu)=r$. If $\Phi_1=0$, then $\Phi_k=0$ for all $1\leq k \leq r$. If $\Phi_1>0$, then $\Phi_k >0$ for all $1\leq k \leq r$. 
\end{lemma}
\begin{proof}
If $\Phi_1=0$, then Lemma~\ref{t-rec} implies $\Phi_k =0$ for all $1\leq k \leq r$. 

By Lemma~\ref{t-rec}, 
\begin{equation}\label{plog}
\Phi_{k+1} \leq 2 t \Phi_{k} +\left(-2+\frac {\eta_{k+1}\cdot \eta_{k-1}} {\eta_k^2} \right) \cdot t^2 \cdot \Phi_{k-1}, 
\end{equation}
where $t=\eta_k\cdot \eta_{k-1}^{-1}\cdot \eta_1^{-1}$. Suppose $\Phi_1>0$ and assume $\Phi_k=0$ for some $k\leq r$. Assume $k$ is the first such index. Then, since 
$$
-2+\frac {\eta_{k+1}\cdot \eta_{k-1}} {\eta_k^2} <0,
$$
by Newton's inequalities for the polynomial 
\begin{equation}\label{pavs}
P(s)= h(\xx+s\uu)/h(\xx) = \sum_{j=0}^r \frac{D_\uu^j h(\xx)} {j!} s^j/h(\xx)= \sum_{j=0}^r \frac{\eta_j} {j!} s^j,   
\end{equation}
\eqref{plog} implies $\Phi_{k+1} <0$, a contradiction. 
\end{proof}

\begin{lemma}\label{stepped}
If $\rk(\uu) = r$, $2\leq k \leq r$ and $\Phi_1 \neq 0$, then 
$$
\frac {\Phi_{k}} {\Phi_{k-1}} \leq \frac k {k-1} \cdot \frac {r-k+2}{r} \cdot  \frac {D_\uu h} h. 
$$
\end{lemma}

\begin{proof}
The proof is by induction over $k\geq 2$. The case when $k=2$ follows immediately from Lemma~\ref{t-rec}. Assume true for all indices  $\leq k$, where $k \geq 2$. By Lemma~\ref{post}, $\Phi_1,\ldots, \Phi_r$ are positive. Let 
$\alpha_j =\eta_1^{-1} \cdot \Phi_j /\Phi_{j-1}$ for $2 \leq j \leq r+1$.  By Lemma~\ref{t-rec}, 
$$
\alpha_{k+1} \leq 2 t_k +\left(-2+\frac {\eta_{k+1}\cdot \eta_{k-1}} {\eta_k^2} \right) \cdot t_k^2 \cdot\alpha_k^{-1}, 
$$
for $2 \leq k \leq r$, 
where $t_k=\eta_k\cdot \eta_{k-1}^{-1}\cdot \eta_1^{-1}$. From Newton's inequalities and \eqref{ncons} for the polynomial \eqref{pavs},
we deduce 
$$
\frac {\eta_{k+1}\cdot \eta_{k-1}} {\eta_k^2} \leq \frac {r-k}{r-k+1} \ \ \mbox{ and } \ \ t_k \leq \frac {r-k+1}{r},
$$
for $1 \leq k \leq r$. 
Hence 
$$
\alpha_{k+1} \leq 2 t_k -\frac {r-k+2}{r-k+1}\cdot t_k^2 \cdot\alpha_k^{-1}, 
$$ 
and by induction (using the bound for $\alpha_k$),
\begin{equation}\label{trik}
\alpha_{k+1} \leq 2 t -\frac {r}{r-k+1}\cdot \frac {k-1}{k} \cdot t^2, \ \ \ \ \ \ \mbox{ where } t=t_k,
\end{equation}
for $2 \leq k \leq r$.
The right-hand-side of \eqref{trik} is increasing in $t$ for 
$$
0\leq t \leq  \frac {r-k+1}{r}\cdot \frac {k}{k-1},
$$
and hence we obtain a valid inequality if we plug in $t= (r-k+1)/r$ in \eqref{trik}, for which we get the desired upper bound.
\end{proof}
\begin{corollary}\label{fk1}
If $\rk(\uu) \leq r$, $1\leq k \leq r$ and $\Phi_1 \neq 0$, then 
$$
\frac {\Phi_{k}} {\Phi_1} \leq k! \binom r {k-1} \cdot \left(\frac {\eta_1} r  \right)^{k-1} . 
$$
\end{corollary}
\begin{proof}
Since $r \mapsto r^{-k+1}\binom r {k-1}$ is increasing we may assume $\rk(\uu) = r$. The lemma now follows by iterating Lemma~\ref{stepped}. 
\end{proof}

 \section{Bounds on zeros of mixed characteristic polynomials}\label{sboundmix}
 Let $h$ be hyperbolic with respect to $\ee$, and let $\vv_1,\ldots, \vv_m \in \Lambda_+(\ee)$. 
 Denote by $\lma(\vv_1,\ldots,\vv_m)$ the largest zero of the \emph{mixed characteristic polynomial}
\begin{equation}\label{mip}
t \mapsto h[\vv_1, \ldots, \vv_m](t\ee+\one)=(1-D_{\vv_1})\cdots (1-D_{\vv_m})h (t\ee), 
\end{equation}
 where $\one \in \RR^m$ is the all ones vector (in the $y$-variables).
 To prove Theorem~\ref{hypprob}, we want to bound $\lma(\vv_1,\ldots,\vv_m)$ conditioned on $\vv_1, \ldots, \vv_m \in \Lambda_+(\ee)$,  $\vv_1+\cdots+\vv_m =\ee$, $\tr(\vv_i) \leq \epsilon$ and $\rk(\vv_i) \leq r$ for all $1\leq i \leq m$. 
 
 \begin{remark}\label{hypid}
Since $\ee$ is in the open hyperbolicity cone $\Gamma_{++}$ of $h[\vv_1, \ldots, \vv_m]$, $\rho\ee +\one$ is also in $\Gamma_{++}$ for $\rho$ sufficiently large. By Theorem~\ref{hypfund} (2), $\rho$ is larger than $\lma(\vv_1,\ldots,\vv_m)$  if and only if $\rho \ee +\one$ is in  $\Gamma_{++}$ of $h[\vv_1, \ldots, \vv_m]$. Consequently 
 $$
 \lma(\vv_1,\ldots,\vv_m)= \inf \{ \rho >0 :  \rho \ee +\one \in \Gamma_{++}\}.
 $$
 \end{remark}
 Next we want to relate $\lma(\vv_1+\cdots+ \vv_m)$ to $\lma(\vv_1,\ldots,\vv_m)$. 
\begin{theorem}\label{linearize}
 If $h$ be hyperbolic with respect to $\ee$ and $\vv_1,\ldots, \vv_m \in \Lambda_+(\ee)$, then 
\begin{align*}
 \lma(\vv_1+\cdots +\vv_m) &\leq  \lma(\vv_1,\ldots,\vv_m), \mbox{ and } \\
\lmi(\vv_1+\cdots +\vv_m) &\geq  \lmi(\vv_1,\ldots,\vv_m). 
\end{align*}
\end{theorem}
\begin{proof}
Consider the polynomial $g(\xx,\yy)=h[\vv_1, \vv_1, \vv_2\ldots, \vv_m]$, which is hyperbolic with respect to $\ee$. Let $\gamma(\zz)$ denote the largest zero of $g(t\ee-\zz)$. Let further $\ee_1,\ldots, \ee_{m+1}$ be the standard basis in the $y$-variables. Then, if $\yy=\ee_3+\cdots+\ee_{m+1}$, 
\begin{align*}
&\lma(\vv_1/2, \vv_1/2, \vv_2,\ldots \vv_m) = \gamma \left(-\frac 1 2 (\ee_1+ \yy)- \frac 1 2 (\ee_2+ \yy) \right) \\
\leq  &\frac 1 2 \gamma\left(-\ee_1- \yy \right) + \frac 1 2 \gamma\left(-\ee_2- \yy \right) = \lma(\vv_1,\vv_2, \ldots, \vv_m),
\end{align*}
by the convexity of $\gamma$, see Theorem~\ref{hypfund}. Iterating this we get that $\lma(\vv_1,\ldots, \vv_m)$ is greater or equal to the largest zero of 
$$
p_N(t) := \prod_{i=1}^m \left(1-\frac {D_{\vv_i}} N\right)^N h(t\ee), 
$$
where $N=2^n$, for any positive integer $n$. However, 
$$
\lim_{N \to \infty} p_N(t)= h(t\ee-\vv_1-\cdots-\vv_m),
$$
and the first statement follows. 

The second statement follows similarly by using the concavity of $\lmi$. 
 \end{proof}
 
\begin{lemma}\label{stayabove}
Let $h$ be hyperbolic, and let $\delta$ and $\mu$ be two positive numbers such that either 
\begin{itemize}
\item $\mu >1$, or 
\item $1\leq \delta \leq 2$ and $\mu >1-\delta/r$. 
\end{itemize}
If $\xx \in \Lambda_{++}$, $\uu \in \Lambda_+$, $0<\rk(\uu)\leq r$ and  $h(\xx)/D_\uu h(\xx) \geq \mu$, then 
$$
(h-D_\uu h)(\xx+\delta \uu) > 0.
$$
\end{lemma}
\begin{proof}
If $\mu>1$, then $h(\xx)/D_\uu h(\xx) >1$ and then $h(\xx+\delta \uu)/D_\uu h(\xx+\delta \uu) \geq h(\xx)/D_\uu h(\xx) >1$, by Theorem~\ref{direct}. Hence $(h-D_\uu h)(\xx+\delta \uu) > 0$. 

Suppose $1\leq \delta \leq 2$ and $\mu >1-\delta/r$. 
We may write 
$$
(h-D_\uu h)(\xx+\delta \uu) = (1-D_\uu)\exp(\delta D_\uu) h(\xx) = \sum_{k= 0}^r a_k \cdot D_\uu^k h(\xx),  
$$
where $a_0=1$ and $a_k= \delta^k/k!-\delta^{k-1}/(k-1)!$ if $k \geq 1$. By \eqref{ajbound}, 
$$
D_\uu^k h(\xx) \leq k! r^{-k}\binom r k h(\xx) \left(\frac {D_\uu h (\xx)}{h(\xx)} \right)^k, \ \ \ 0 \leq k \leq r,
$$ 
with equality for $k=0,1$. Now $a_0 =1, a_1 \geq 0$ and $a_k \leq 0$ for $2 \leq k \leq r$, 
since $1\leq \delta \leq 2$. Hence 
\begin{align*}
(h-D_\uu h)(\xx+\delta \uu) &\geq \sum_{k= 0}^r a_k \cdot k! r^{-k}\binom r k h(\xx) \mu^{-k} \\
&= h(\xx)\cdot \left(1+ \frac {\delta}{\mu r} \right)^{r-1}\cdot \left(1+\frac {\delta}{\mu r}-\frac 1 \mu \right),
\end{align*}
and the lemma follows. 
\end{proof}

 For the remainder of this section, let $h \in \RR[x_1,\ldots, x_n]$ be hyperbolic with respect to $\ee$, and let $\vv_1,\ldots, \vv_m \in \Lambda_{+}$ and $\xx \in \Lambda_{++}$. 
To enhance readability, 
let $\partial_j := D_{\vv_j}$. 
and 
$$
\xi_j[g] := \frac {g}{\partial_j g}.
$$

\begin{lemma}\label{eng2}
Suppose $(\delta, \mu) \in U_r$, where $r$ is a positive integer or $r=\infty$. 
If $\xx \in \Lambda_{++}$, $0<\rk(\vv_j) \leq r$, $0<\rk(\vv_i)$ and 
$$
\xi_j[h](\xx) \geq \mu,
$$ 
then 
$$
\xi_i[h-\partial_j h](\xx+\delta \vv_j) \geq \xi_i[h](\xx). 
$$
\end{lemma}

\begin{proof}
Suppose $\xi_j[h](\xx) \geq \mu$ and $h$ is normalized so that $h(\ee) >0$. Write
$$
(h-\partial_j h)(\xx+\delta \vv_j) = (1-\partial_j)\exp(\delta \partial_j) h(\xx) = \sum_{k\geq 0} a_k \cdot \partial_j^k h(\xx),  
$$
where $a_0=1$ and $a_k= \delta^k/k!-\delta^{k-1}/(k-1)!$ if $k \geq 1$. By Lemma~\ref{stayabove}, $(h-\partial_j h)(\xx+\delta \vv_j)>0$, and then also $(\partial_ih-\partial_j \partial_ih)(\xx+\delta \vv_j)>0$ by Theorem~\ref{direct}.  
We want 
$$
\frac {\sum_{k\geq 0} a_k \cdot \partial_j^k h} {\sum_{k\geq 0} a_k \cdot \partial_i \partial_j^k h} \geq \frac h {\partial_i h},
$$
that is 
\begin{equation}\label{tronk}
\sum_{k\geq 1} a_k \cdot (\partial_j^k h \cdot \partial_i h -\partial_i\partial_j^k h \cdot h ) \geq 0. 
\end{equation}
Recall the definition of $\Phi_k$. For $\vv_i=\vv$ and $\vv_j=\uu$, \eqref{tronk} amounts to 
\begin{equation}\label{lala0}
\sum_{k\geq 1} a_k \cdot \Phi_k \geq 0. 
\end{equation}
If $\Phi_1=0$, then \eqref{lala0} holds by Lemma~\ref{post}.
If $\Phi_1 >0$, then 
for $\delta \geq 1$, \eqref{lala0} is equivalent to
\begin{equation}\label{lala}
\delta-1 \geq \left(\delta - \frac {\delta^2} 2 \right)\cdot \frac {\Phi_2}{\Phi_1} + \left(\frac {\delta^2} 2 - \frac {\delta^3} 6 \right)\cdot \frac {\Phi_3}{\Phi_1}+ \cdots  + \left(\frac {\delta^{k-1}} {(k-1)!} - \frac {\delta^k} {k!} \right)\cdot \frac {\Phi_k}{\Phi_1}+\cdots
\end{equation}
or, equivalently, 
$$
\delta-1\geq \left({\sum_{k \geq 1} (k-1)\frac {\Phi_k}{\Phi_1} \frac {\delta^{k-1}} {k!}}\right) \Big /\left({\sum_{k \geq 1} \frac {\Phi_k}{\Phi_1} \frac {\delta^{k-1}} {k!}}\right). 
$$
Since $0\leq \delta-1 \leq 1$, all terms in \eqref{lala} are nonnegative. If we use Corollary \ref{fk1} to replace ${\Phi_k}/{\Phi_1}$ with  $k! \binom r {k-1} \cdot \left( 1/\mu r  \right)^{k-1}$ in \eqref{lala} we get the inequality \eqref{dura}. 
\end{proof}

\begin{corollary}\label{corbond2}
Suppose $h$ is hyperbolic with respect to $\ee \in \RR^n$, and let $\Gamma_+$ be the (closed) hyperbolicity cone of $h[\vv_1,\ldots, \vv_m]$, where $\vv_1,\ldots, \vv_m \in \Lambda_{+}(\ee)$ and $1\leq \rk(\vv_k) \leq r_k$ for all $1\leq k \leq m$. 
Suppose  $\xx \in \Lambda_{++}(\ee)$ and $1\leq i<j\leq m$ are such that 
$$
\xx+\mu_k \ee_k \in \Gamma_+, \quad \mbox{ for } k \in \{i,j\},
$$
where $\mu_i,\mu_j>0$. 
Then 
$$
\xx+\delta_j\vv_j + \ee_j + \mu_i \ee_i \in \Gamma_+,
$$ 
whenever $(\delta_j, \mu_j) \in U_{r_j}$. 

Moreover if  $\xx+\mu_k \ee_k \in \Gamma_+$, where $(\delta_k, \mu_k) \in U_{r_k}$ for all $k \in [m]$, then
$$
\xx+ \left(1-\frac 1 m\right) \sum_{i=1}^m \delta_i \vv_i +\left(1-\frac 1 m\right)\sum_{i=1}^m \ee_i+ \frac 1 m\sum_{i=1}^m \mu_i\ee_i \in \Gamma_+. 
$$
\end{corollary}
\begin{proof}
Recall that 
$$ 
\xx + \mu_k \ee_k \in \Gamma_+ \mbox{ if and only if } \xi_k[h] \geq \mu_k.
$$
By Lemma~\ref{stayabove} and \ref{eng2},  $\xx + \ee_j \in \Gamma_+$  and 
$$
\xi_i[h-\partial_jh](\xx+\delta_j \vv_j) \geq \mu_i, 
$$ 
which is equivalent to  
$$
\xx+\delta_j\vv_j + \ee_j+\mu_i \ee_i \in \Gamma_+.
$$
Hence the first part follows. 

Suppose $\xx+\mu_k \ee_k \in \Gamma_+$ for all $k \in [m]$. Since $\xx+s\ee_1,  \vv_1 \in \Gamma_+$ for all $s \leq \mu_1$, the vector 
$$
\xx' := \xx+\delta_1\vv_1 + \ee_1
$$
is in the hyperbolicity cone of $(1-y_1D_{\vv_1})h$. 
By the first part we have $\xx'+\mu_2\ee_2, \xx'+\mu_3\ee_3\in \Gamma_+$. Hence we may apply the first part of the theorem with $h$ replaced by $(1-y_1D_{\vv_1})h$ to conclude 
$$
\xx'+ \delta_2\vv_2 + \ee_2+ \mu_3\ee_3=\xx+\delta_1\vv_1 + \delta_2 \vv_2+\ee_1 +\ee_2 + \mu_3\ee_3\in \Gamma_+. 
$$
By continuing this procedure with different orderings we may conclude that 
$$
\xx +  \left(\sum_{i=1}^m \delta_i\vv_i\right)-\delta_j\vv_j  +\left(\sum_{i=1}^m \ee_i\right)-\ee_j+\mu_j\ee_j \in \Gamma_+,
$$
for each $1\leq j \leq m$. The second part now follows from convexity of $\Gamma_+$ upon taking the convex sum of these vectors.
\end{proof}

\begin{theorem}\label{mainbound2}
Suppose $h$ is hyperbolic with respect to $\ee \in \RR^n$ and suppose $\vv_1,\ldots, \vv_m \in \Lambda_{+}(\ee)$ are of rank at most $r$ and such that $\ee = \vv_1+\cdots+\vv_m$, where $\tr(\vv_j) \leq \epsilon$ for each $1\leq j \leq m$. Then
$$
\lma(\vv_1,\ldots, \vv_m) \leq \delta(\epsilon, m,r),
$$
where
$$
\delta(\epsilon, m,r) = \inf \left\{ \frac {\epsilon \mu+ \left(1-\frac 1 m\right)\delta} {1-\frac 1 m + \frac \mu m }  : (\delta, \mu) \in U_r \right\}.
$$
\end{theorem}
\begin{proof}
For $\mu >0$, set $\xx=\epsilon \mu\ee$ and $\mu_i=\mu$ for $1\leq i \leq m$.  Then 
$\xx+\mu_i \ee_i = \mu(\epsilon \ee+ \ee_i) \in \Lambda_+$ since 
$$
h[\vv_1,\ldots, \vv_m](\epsilon \ee + \ee_i)= \epsilon h(\ee)- D_{\vv_i}h(\ee) = h(\ee)(\epsilon -\tr(\vv_i)) \geq 0. 
$$
Apply Corollary \ref{corbond2} to conclude that :
$$
\left(\epsilon \mu+ \left(1-\frac 1 m\right)\delta\right) \ee + \left(1-\frac 1 m + \frac \mu m \right) \one \in \Gamma_+, 
$$
whenever $(\delta, \mu)\in U_r$. 
Hence by (the homogeneity of $\Gamma_+$ and) Remark \ref{hypid}, the maximal zero is at most $\delta(\epsilon, m,r)$. 
\end{proof}

In Section \ref{sspec} we compute $\delta(\epsilon, m,r)$ for special cases and prove an upper bound.

\section{Proof of the main theorem}\label{sproof}
To prove Theorem~\ref{t1} we use the following theorem.
\begin{theorem}\label{hypprob}
Suppose $h$ is hyperbolic  with respect to $\ee$. Let $\XR_1, \ldots, \XR_m$ be independent random vectors in $\Lambda_+(\ee)$  with finite supports  such that 
\begin{equation}\label{hypeta2}
\sum_{i=1}^m \EE\XR_i =\ee, 
\end{equation}
\begin{equation}\label{hyptr}
\tr(\EE \XR_i) \leq \epsilon \mbox{ for all } 1\leq i \leq m,
\end{equation}
and 
$$
\rk(\EE \XR_i) \leq r \mbox{ for all } 1\leq i \leq m,
$$
then 
\begin{equation}\label{hypbig}
\PP\left[ \lma\left(\sum_{i=1}^m \XR_i \right) \leq \delta(\epsilon,m,r) \right] >0.
\end{equation}

\begin{proof}
Let $V_i$ be the support of $\XR_i$, for each $1 \leq i \leq m$. By Theorem~\ref{mixedchar}, the family 
$$
\{h[\vv_1,\ldots, \vv_m](t\ee+\one)\}_{\vv_i \in V_i}
$$
is compatible. By Theorem~\ref{expfam} there are vectors $\vv_i \in V_i$, $1\leq i \leq m$, such that the largest zero of $h[\vv_1,\ldots, \vv_m](t\ee+\one)$ is smaller or equal to the largest zero of 
$$
 \EE h[\XR_1,\ldots, \XR_m](t\ee+\one)= h[\EE \XR_1,\ldots, \EE \XR_m](t\ee+\one).
$$
In other words,  there are vectors $\vv_i \in V_i$, $1\leq i \leq m$, such 
$$
\lma(\vv_1,\ldots, \vv_m) \leq \lma(\EE \XR_1,\ldots, \EE \XR_m)
$$
The theorem now follows from Theorem~\ref{mainbound2} and Theorem~\ref{linearize}. 
\end{proof}

\end{theorem}

\begin{proof}[Proof of Theorem~\ref{t1}]
For $1\leq i \leq k$, let $\xx^i=(x_{i1},\ldots,x_{in})$ where $\yy=\{x_{ij} : 1\leq i \leq k, 1\leq j \leq k\}$ are independent variables. Consider the polynomial 
$$
g(\yy) = h(\xx^1)h(\xx^2) \cdots h(\xx^k) \in \RR[\yy],
$$
which is hyperbolic with respect to $\ee^1\oplus \cdots \oplus \ee^k$, where $\ee^i$ is a copy of $\ee$ in the variables $\xx^i$, for all $1 \leq i \leq k$. The hyperbolicity cone of $g$ is the direct sum $\Lambda_+:=\Lambda_+(\ee^1) \oplus \cdots \oplus \Lambda_+(\ee^k)$, where $\Lambda_+(\ee^i)$ is a copy of $\Lambda_+(\ee)$ in the variables $\xx^i$,   for all $1 \leq i \leq k$. 

Let $\XR_1, \ldots, \XR_m$ be independent random vectors in $\Lambda_+$  such that for all $1\leq i \leq k$ and  $1\leq j \leq m$:
$$
\PP\left[ \XR_j = k\uu_j^i\right] =  \frac 1 k,
$$
where $\uu_1^i, \ldots, \uu_m^i$ are copies in $\Lambda_+(\ee^i)$ of $\uu_1, \ldots, \uu_m$. Then 
\begin{align*}
\EE \XR_j &= \uu_j^1 \oplus \uu_j^2 \oplus \cdots \oplus \uu_j^k, \\
\tr(\EE \XR_j) &= k\tr(\uu_j) \leq k\epsilon,  \\
\rk(\EE \XR_j) &= k\rk(\uu_j) \leq kr, \mbox{ and } \\
\sum_{j=1}^m \EE \XR_j &= \ee^1\oplus \cdots \oplus \ee^k,
\end{align*}
for all $1\leq j \leq m$. By Theorem~\ref{hypprob} there is a partition 
$S_1\cup \cdots \cup S_k =[m]$ such that 
$$
\lma\left(\sum_{i \in S_1}k\uu_i^1+\cdots + \sum_{i \in S_k}k\uu_i^k  \right)\leq \delta(k\epsilon,m,kr). 
$$
However 
$$
 \lma\! \left(\sum_{i \in S_1}k\uu_i^1+\cdots + \sum_{i \in S_k}k\uu_i^k \right) = k \! \max_{1\leq j \leq k} \lma \! \left(\sum_{i \in S_j}\uu_i^j \right) = 
 k \! \max_{1\leq j \leq k} \lma \! \left(\sum_{i \in S_j}\uu_i \right), 
$$
and the theorem follows.
\end{proof}

\section{Specific bounds}\label{sspec}
Finally we prove the specific bounds in Theorem~\ref{specifics}. 
\begin{proof}[Proof of Theorem~\ref{specifics}]
The infimum for the case when $r=\infty$ i.e., \eqref{a0} and \eqref{a1}, is easily computed. 

For \eqref{a2}, we first compute the infimum $\alpha$ when $1\leq \delta \leq 2$ and $\mu>1-\delta/2$. For $r=2$, \eqref{dura} simplifies to 
$$
\mu \geq (\delta-1)^{-1}-(\delta-1).
$$
Since $1-\delta/2 \leq (\delta-1)^{-1}-(\delta-1)$,
$$
\alpha = \inf\{ \epsilon ( (\delta-1)^{-1}-(\delta-1)) + \delta : 1< \delta \leq 2\} 
$$
If $0 < \epsilon < 1$, the function $\delta \mapsto \epsilon ( (\delta-1)^{-1}-(\delta-1))+\delta$, $\delta > 1$, has a unique minimum at $\delta_0 = 1+\sqrt{\epsilon}/ \sqrt{1-\epsilon}$. Since $1<\delta_0 \leq 2$ if and only if $\epsilon \leq 1/2$, 
$$
\alpha =
\begin{cases}
1 + 2\sqrt{\epsilon}\sqrt{1-\epsilon}, &\mbox{ if } 0<\epsilon \leq 1/2 \\
2, &\mbox{ if } \epsilon >1/2. 
\end{cases}
$$
Let $\beta$ be the infimum when $\delta \geq 2$ and $\mu>1$. Then $\beta\geq  \epsilon +2$, and hence $\alpha < \beta$. 

For \eqref{a3}, note that for $r \geq 1$ and $x\geq 0$, 
$$
\frac { (x+1)^{r-1} -x^{r-1} } {(x+1)^r-x^r} \leq \frac 1 {1+x}.
$$
Hence to get an upper bound for $\delta(\epsilon,\infty,r)$ we may replace \eqref{dura} with 
$$
\delta-1 \geq \frac \delta \mu \cdot \frac { 1} { 1+ \frac {\delta} {r \mu}}
$$
or equivalently
$$
\mu \geq 1+ \frac 1 {\delta-1} -\frac \delta r.
$$
Thus the inequality $\mu \geq 1 -\frac \delta r$ is superfluous, so that 
$$
\delta(\epsilon,\infty,r) \leq \inf \left\{ \epsilon \mu + \delta : 1<\delta \leq 2, \mu \geq 1+ \frac 1 {\delta-1} -\frac \delta r\right\},
$$
which is (computed as above and) equal to
$$
\begin{cases}
1+2\sqrt{\epsilon}\sqrt{1-\epsilon/r}+ \frac {r-1} r \epsilon, &\mbox{ if } \epsilon \leq r/(r+1), \\ 
2+\epsilon(1-2/r), &\mbox{ if } \epsilon > r/(r+1).
\end{cases}
$$
\end{proof}


\begin{thebibliography}{99}
\bibitem{AB} N.~Amini, P.~Br\"and\'en, Non-representable hyperbolic matroids, Adv. Math. {\bf 334} (2018), 417--449.
\bibitem{ABG} M.~F.~Atiyah, R.~Bott, L~G\aa rding, Lacunas for hyperbolic differential operators with constant coefficients. I, Acta Math. {\bf 124} (1970), 109-189.


\bibitem{BGLS} H.~H.~Bauschke, O.~G\"uler, A.~S.~Lewis, H.~S.~Sendov,  Hyperbolic polynomials and convex analysis, Canad. J. Math. {\bf 53} (2001), 470-488.

\bibitem{BBL} J.~Borcea, P.~Br\"and\'en,  T.~M.~Liggett, {Negative dependence and the geometry of polynomials}, J. Amer.
Math. Soc. {\bf 22} (2009), 521-567. 

\bibitem{R2}
M.~Bownik, P.~G.~Casazza, A.~W.~Marcus, D.~Speegle, 
Improved bounds in Weaver and Feichtinger Conjectures,  J. Reine Angew. Math. (to appear).

\bibitem{BrObs} P.~Br\"and\'en, { Obstructions to determinantal representability,} Adv. Math., {\bf 226} (2011), 1202--1212. 
%
\bibitem{BrOp} P.~Br\"and\'en, Hyperbolicity cones of elementary symmetric polynomials are spectrahedral, Optim. Lett. {\bf 8} (2014), 1773-1782. 

\bibitem{BLec0} P.~Br\"and\'en, Lecture notes for Interlacing families, \url{https://people.kth.se/~pbranden/notes.pdf}

\bibitem{BLec} P.~Br\"and\'en, Hyperbolic polynomials and the Marcus-Spielman-Srivastava theorem, Lecture notes, \url{https://arxiv.org/abs/1412.0245}. 


%
%





\bibitem{Cas} P.~G.~Casazza, J.~C.~Tremain, Consequences of the Marcus/Spielman/Srivastava solution of the Kadison-Singer problem. New trends in applied harmonic analysis, 191--213, Appl. Numer. Harmon. Anal., BirkhŠuser/Springer, Cham, 2016.  

\bibitem{CS} M.~Chudnovsky, P.~Seymour, The roots of the independence polynomial of a clawfree graph, J. Combin. Theory Ser. B {\bf 97} (2007),  350-357.

\bibitem{Cohen} M.~Cohen, \url{https://open.library.ubc.ca/cIRcle/collections/48630/items/1.0340957}


\bibitem{Ga} 
L.~G\aa rding, {An inequality for hyperbolic polynomials}, 
J. Math. Mech. {\bf 8} (1959), 957-965. 

%
%
%
%

\bibitem{Horm} L.~H\"ormander, The analysis of linear partial differential operators. II. Differential operators with constant coefficients, Springer-Verlag, Berlin, 1983.


\bibitem{KS} R.~V.~Kadison, I.~M.~Singer, Extensions of pure states, Amer. J. Math. {\bf 81} (1959),  383-400.


%


\bibitem{MSS1} A.~W.~Marcus, D.~A.~Spielman, N.~Srivastava, Interlacing families I: Bipartite Ramanujan graphs of all degrees, Ann. of Math. (2) {\bf 182} (2015), no. 1, 327--350.

\bibitem{MSS2} A.~W.~Marcus, D.~A.~Spielman, N.~Srivastava, Interlacing families II: Mixed characteristic polynomials and the Kadison-Singer problem, Ann. of Math. (2) {\bf 182} (2015), no. 1, 307--325.  

\bibitem{Newton} I.~Newton, Arithmetica universalis: sive de compositione et resolutione arithmetica liber (1707).

\bibitem{Pem} R.~Pemantle, Hyperbolicity and stable polynomials in combinatorics and probability, Current developments in mathematics, 2011, 57--123, Int. Press, Somerville, MA, 2012. 

\bibitem{Ren} J.~Renegar,  Hyperbolic programs, and their derivative relaxations,  Found. Comput. Math., 
{\bf 6} (2006), 59--79.

\bibitem{Vin} V.~Vinnikov, 
LMI representations of convex semialgebraic sets and determinantal representations of algebraic hypersurfaces: past, present, and future,  Mathematical methods in systems, optimization, and control, 325--349, 
Oper. Theory Adv. Appl., {\bf 222}, BirkhŠuser/Springer Basel AG, Basel, 2012. 

\bibitem{Wag} D.~G.~Wagner, Multivariate stable polynomials: theory and applications, Bull. Amer. Math. Soc. {\bf 48} (2011), 53--84. 

\bibitem{We} N.~Weaver, 
The Kadison-Singer problem in discrepancy theory,
Discrete Math. {\bf 278} (2004),  227-239.

\end{thebibliography}
\end{document}